\documentclass[reqno]{amsart}
\usepackage{amssymb}
\usepackage[usenames, dvipsnames]{color}

\usepackage{hyperref}
\usepackage{verbatim}

\theoremstyle{plain}
\newtheorem{theorem}{Theorem}[section]
\newtheorem{lemma}[theorem]{Lemma}

\newtheorem{proposition}[theorem]{Proposition}

\theoremstyle{definition}

\theoremstyle{remark}
\newtheorem{remark}[theorem]{Remark}

\newcommand{\bR}{{\mathbb R}}
\newcommand{\bN}{{\mathbb N}}

\newcommand{\bQ}{{\mathbb Q}}
\newcommand{\bP}{{\mathbb P}}
\def\ve{\varepsilon}
\def\la{\lambda}
\def\La{\Lambda}
\def\t{\tilde}
\def\q{\quad}
\def\qq{\qquad}

\def\g{\gamma}
\def\G{\Gamma}
\def\dl{\delta}
\def\Dl{\Delta}
\def\lt{\left}

\def\rt{\right}

\def\b{\beta}
\def\i{\infty}

\def\Div{\text{div }}
\def\sgn{\text{sgn }}
\def \ls{\lesssim}
\def\p{\partial}
\def\f{\frac}
\def\na{\nabla}
\def\al{\alpha}

\def\O{\Omega}

\def\s{\sqrt}

\numberwithin{equation}{section}

\makeatletter
\def\dashint{\operatorname%
{\,\,\text{\bf--}\kern-.98em\DOTSI\intop\ilimits@\!\!}}
\makeatother

\begin{document}
\title[Time analyticity]{Time analyticity for inhomogeneous parabolic equations and the Navier-Stokes equations in the half space}

\author[H. Dong]{Hongjie Dong}
\address[H. Dong]{Division of Applied Mathematics, Brown University, 182 George Street, Providence, RI 02912, USA}

\email{Hongjie\_Dong@brown.edu}


\author[X. Pan]{Xinghong Pan}
\address[X. Pan]{Department of Mathematics, Nanjing University of Aeronautics and Astronautics, Nanjing 211106, China}

\email{xinghong\_87@nuaa.edu.cn}

\thanks{X. Pan is supported by Natural Science Foundation of Jiangsu Province (No. SBK2018041027) and National Natural Science Foundation of China (No. 11801268).}

\subjclass[2010]{time analyticity, parabolic equations, the Navier-Stokes equations}

\keywords{35K10, 35Q30}

\begin{abstract}
We prove the time analyticity for weak solutions of inhomogeneous parabolic equations with measurable coefficients in the half space with either the Dirichlet boundary condition or the conormal boundary condition under the assumption that the solution and the source term have the exponential growth of order $2$ with respect to the space variables. We also obtain the time analyticity for bounded mild solutions of the incompressible Navier-Stokes equations in the half space with the Dirichlet boundary condition. Our work is an extension of the recent work in \cite{DZ:2019arxiv1, Zq:2019arxiv1}, where the authors proved the time analyticity of solutions to the homogeneous heat equation and the Navier-Stokes equations in the whole space.
\end{abstract}
\maketitle

\section{Introduction}

For parabolic equations, it is well known that solutions are analytic in the space variables under reasonable conditions on the coefficients and data. In fact, the space analyticity is a local property, meaning that to show the space analyticity at a given point, we only need to impose conditions in a neighborhood of it.
In contrast, the time analyticity of solutions is a more delicate issue and is false in general. For example, it is not difficult to construct a solution of the heat equation in a finite space-time cylinder, which is not time analytic in a sequence of moments.
The time analyticity is not a local property, so we need to impose certain growth conditions on solutions and data at infinity.
Under additional assumptions, there are many time-analyticity results for the heat equation and other parabolic type equations. See, for example, \cite{Wdv:1962DMJ, Mk:1967PJA, Kg:1980CPAM, Gy:1983CPDE, EMZ:2017JMA}.

In a related development, there have been increasing interest in the study of ancient solutions of parabolic equations, solutions that exist for all negative time. In \cite{SZ:2006BLMS}, the authors proved that sublinear ancient solutions are constants. Later, it was shown in \cite{LZ:2019CPAM} that the space dimension of ancient solutions of polynomial growth is finite and these solutions are polynomials in time. In \cite{CM:2019ARXIV}, a sharp dimension estimate was established. See also the papers \cite{Cm:2006MZ,CM:2019ARXIV1} for applications in the study of mean curvature flow on manifolds, and \cite{Hb:2019ARXIV} for the graph case. In a recent paper \cite{Zq:2019arxiv1}, the author observed that the ancient solution of heat equations with exponential growth with respect to the space variables is analytic in time. This result was extended in \cite{DZ:2019arxiv1} to solutions with exponential growth of order $2$ with respect to the space variables, which is a sharp condition. Moreover, the time analyticity of bounded mild solutions for the incompressible Navier-Stokes equations was also proved in \cite{DZ:2019arxiv1} by using a real-variable argument.

The goal of this paper is to extend the result in \cite{DZ:2019arxiv1} and \cite{Zq:2019arxiv1} to inhomogeneous parabolic equations and the Navier-Stokes equations in the half space. More precisely, our first main result is the time analyticity for weak solutions of inhomogeneous parabolic equations with time-independent measurable coefficients in the half space with either the Dirichlet or the conormal boundary condition, under the assumption that the solution and the source term have the exponential growth of order $2$ with respect to the space variables. For the proof, we first reformulate the problem to the whole space case by using the odd and even extensions. Then we apply an iteration argument used in \cite{DZ:2019arxiv1} to estimate high-order time derivatives of the solution. See \eqref{etds} below.
We note that the method here can be applied to more general linear or nonlinear parabolic equations with inhomogeneous source terms.


For the Navier-Stokes equations, the space analyticity of solutions has been studied extensively in the literature. See, for example, \cite{Kc:1969ARMA,GK:1998JFA,GPS:2007IMRN,DL:2009CMS,BBT:2012ARMA,CKV:2018JDE,Xl:2018ARXIV} and the references therein. There are also many work regarding the time analyticity for the Navier-Stokes equations. In a four-page note \cite{CN:1}, the authors proved that any bounded mild solution of the 3D Navier-Stokes equations is time analytic if the gradient of the solution and the pressure have sublinear growth with respect to the space variables and the solution converges a constant vector as $x\to\infty$. The time analyticity with values in an $L_2$-based Gevrey
class of periodic functions was proved for the Navier-Stokes equations in
\cite{FT:1}.
In \cite{Gy:1983CPDE}, the time analyticity was obtained for any weak solution in $C((0,T),W^{1,p}(\O))$ for $n/2<p<\i$ in a bounded smooth domain with the Dirichlet boundary condition. Later, Giga-Jo-Mahalov-Yoneda \cite{GJMY:2008PD} proved the time analyticity of bounded uniformly continuous mild solutions to the Navier-Stokes equations with the Coriolis force and spatially almost periodic data. In \cite{Ki:1991JDDE}, the author studied the analyticity radius of the space periodic solutions of the 2D Navier-Stokes equations. We point out that the proofs in these papers are based on a complexification argument.
In a recent interesting paper, by using a direct energy-based method, Camliyurt-Kukavica-Vicol \cite{CKV:2018JDE} established the instantaneous space-time analyticity and Gevrey regularity for the Navier-Stokes equations in the half space under the assumption that the initial data belongs to $H^1_0(\bR^n_+)\cap H^4(\bR^n_+)$ and satisfies suitable compatibility conditions.   

Our second main objective of this paper is to prove the time analyticity of the bounded mild solution for the Navier-Stokes equations in the half space with the Dirichlet boundary condition. Compared to the results in \cite{CN:1,Gy:1983CPDE,Ki:1991JDDE}, we consider the equations in the half space and we do not impose any gradient and pressure controls, periodicity or the global integrability assumptions on the solutions. Also, in contrast to \cite{GJMY:2008PD}, we prove the time analyticity by using a real-variable argument and we do not impose any (almost) periodicity condition. 
Finally, different from \cite{CKV:2018JDE} in which solutions are assumed to be in the energy space, we consider bounded and continuous solutions which may not decay at the space infinity.
Let us briefly describe the method of our proof. First we recall the formulation of mild solutions in the half space, which is given by convolutions of certain kernels with the initial data, the source term, and the square of the solution. See \eqref{milds} below. These kernels have certain integrability properties, which ensure the local-in-time solvability of the initial value problem. In order to prove the time analyticity, we estimate higher order time derivatives of the solution. Here the difficulty is that we cannot directly take the time derivatives of the kernels because their derivatives are in general not integrable in space-time. To this end, we follow a technique used in \cite{DZ:2019arxiv1} by using an algebraical manipulation of the kernels. Finally, by induction we bound the $k$-th order time derivative of the bounded mild solution by $M^{k+1}k^k t^{-k}$ for any $t>0$ for some large constant $M$, which implies the time analyticity of the solution.

We will present our results in Section 2 for linear parabolic equations and in Section 3 for the Navier-Stokes equations, respectively. The symbol $ ... \ls ...$ stands for $... \le C ...$ for a positive constant $C$. The notation $C$ with or without indices denotes a positive constant whose value may change from line to line. We use $\p_i$ to denote $\p_{x_i}$ for $i=1,2,\ldots,n$. Throughout the paper, the summation convention over repeated indices are used. Let $(t_0,x_0)$ be a given space-time point in $\bR^{n+1}$. We denote the parabolic cylinder in $\bR^{n+1}$ by
\begin{equation*}
Q_r(t_0,x_0):=\lt\{(t,x):\ |x-x_0|< r,\ t\in(t_0-r^2,t_0)\rt\}
\end{equation*}
and the space ball by
\begin{equation*}
B_r(x_0):=\lt\{x:\ |x-x_0|< r\rt\}.
\end{equation*}
Sometimes, we will ignore the center point $(t_0,x_0)$ to denote $Q_r(t_0,x_0)$ by $Q_r$ and $B_r(x_0)$ by $B_r$ if no confusion is caused. We also write $x=(x',x_n)$ with $x'=(x_1,x_2,\ldots,x_{n-1})\in \bR^{n-1}$.

\section{Inhomogeneous parabolic equations in the half space}
In this section, we consider the time analyticity of the divergence form parabolic equations
\begin{equation}
            \label{ep1}
\p_t u-\p_i(a_{ij}(x)\p_j u)=\p_if_i(t,x)\q \text{in}\ (-2,0]\times\bR^n_+
\end{equation}
with the Dirichlet boundary condition
\begin{equation}
        \label{eDirichlet}
u|_{x_n=0}=0
\end{equation}
or the conormal boundary condition
\begin{equation}
        \label{eNeumann}
(a_{nj}\p_j u+f_n)|_{x_n=0}=0.
\end{equation}
Here $(a_{ij})_{1\leq i,j\leq n}$ are bounded measurable functions and satisfy the uniform ellipticity condition, \textit{i.e.}, there exist two constants $0<\la<\La<+\infty$ such that
\begin{equation}
        \label{euelliptic}
\la |\xi|^2\leq a_{ij}(x)\xi_i\xi_j,\quad|a_{ij}|\leq \La \q \text{for}\ \xi=(\xi_1,\ldots,\xi_n)\in\bR^n.
\end{equation}

We state the main result of this section as the following theorem.
\begin{theorem}\label{th1}
Let $u$ be a weak solution of \eqref{ep1} with boundary condition \eqref{eDirichlet} or \eqref{eNeumann}.
Assume that
\begin{equation}
        \label{ebu}
|u(t,x)|\leq A_1 e^{A_2|x|^2}\q \text{in}\ (-2,0]\times \bR^n_+,
\end{equation}
and the source term $f_i$ satisfies
\begin{equation}
        \label{ebc}
|\p^k_t f_i(t,x)|\leq A_1 C^kk^k e^{A_2|x|^2}\q \text{in}\ (-2,0]\times \bR^n_+, \q k\in \{0\}\cup\bN,\ 1\leq i\leq n.
\end{equation}
Then $u=u(t,x)$ is analytic in time at any point $(t_0,x)$ for $t_0\in[-1,0]$ with the radius $\dl$ depending only on $n$, $A_2$, $\la$, and $\La$. Moreover, we have for any $t\in [t_0- \dl,t_0+\dl]$,
\begin{equation*}
u(t,x)=\sum^{+\infty}_{j=0}d_j(t_0,x)\f{(t-t_0)^j}{j!}
\end{equation*}
with
\begin{equation*}
|d_j(t_0,x)|\leq A_1 A^{j+1}_3j^je^{2A_2|x|^2},
\end{equation*}
where $A_3$ depends only on $n$, $A_2$, $\la$, and $\La$.
\end{theorem}
\begin{remark}
The result in Theorem \ref{th1} also holds for more general linear parabolic equations
\begin{equation*}
\p_t u-\p_i(a_{ij}(x)\p_ju)+b_i(x)\p_i u+\p_i (\tilde b_i(x) u)+c(x)u=f(t,x)+\p_if_i(t,x)
\end{equation*}
with bounded and measurable coefficients $b_i$, $\tilde b_i$, and $c$, and data $f$ satisfying suitable growth condition. In this paper we do not make this generalization since the essential idea of the proof is the same as that in Theorem \ref{th1}.
\end{remark}
The first step of our proof is to extend our problem to be the one in the whole space by using odd and even extensions. This is possible because we do not impose any regularity assumption on the coefficients and data with respect to $x$.

\subsection{Reformulation in the whole space}
We take different extensions for the Dirichlet and conormal boundary conditions.

\noindent{\bf Case 1:} The Dirichlet boundary condition.

If the boundary condition is given by \eqref{eDirichlet}, we make the following extension. For $(t,x)\in (-2,0]\times\bR^n$, let
\begin{equation*}
\t{u}(t,x)= \sgn(x_n)u(t,x',|x_n|),
\end{equation*}
\begin{equation*}
\t{f}_i(t,x)=\lt\{
\begin{aligned}
&\sgn(x_n)f_i(t,x',|x_n|),\q i\neq n,\\
&f_i(t,x',|x_n|),\qq\qq i=n,
\end{aligned}
\rt.
\end{equation*}
and
\begin{equation*}
\t{a}_{ij}(t,x)=\lt\{
\begin{aligned}
&\sgn(x_n)a_{ij}(t,x',|x_n|),\q i\neq n,j=n\ \text{or}\ i=n,\ j\neq n,\\
&a_{ij}(t,x',|x_n|),\qq\qq \text{otherwise}.
\end{aligned}
\rt.
\end{equation*}

\noindent{\bf Case 2:} The conormal boundary condition.

If the boundary condition is given by \eqref{eNeumann}, we make the following extension. For $(t,x)\in (-2,0]\times \bR^n$, let
\begin{equation*}
\t{u}(t,x)=u(t,x',|x_n|),
\end{equation*}
\begin{equation*}
\t{f}_i(t,x)=\lt\{
\begin{aligned}
&f_i(t,x',|x_n|),\qq\qq i\neq n,\\
&\sgn(x_n)f_i(t,x',|x_n|),\q i=n,
\end{aligned}
\rt.
\end{equation*}
and
\begin{equation*}
\t{a}_{ij}(t,x)=\lt\{
\begin{aligned}
&\sgn(x_n)a_{ij}(t,x',|x_n|),\q i\neq n,j=n\ \text{or}\ i=n,\ j\neq n,\\
&a_{ij}(t,x',|x_n|),\qq\qq \text{otherwise}.
\end{aligned}
\rt.
\end{equation*}

In both cases, it is easily seen that $\t{u}$ is a weak solution of the following equation in the whole space $(-2,0]\times\bR^n$:
\begin{equation}
        \label{ewhole}
\p_t \t{u}-\p_i(\t{a}_{ij}(x)\p_j \t{u})=\p_i\t{f}_i(t,x)
\end{equation}
with $\t{a}_{ij}$ satisfying the assumption \eqref{euelliptic} and $\t{u}$ and $\t{f}_i$ satisfying the assumptions \eqref{ebu} and \eqref{ebc} in Theorem \ref{th1}. Moreover, according to our extensions, $\t{u}\equiv u$ in the half space $(-2,0]\times\bR^n_+$. Later on, we will prove the time analyticity for solutions of \eqref{ewhole}. For simplicity, we will drop the tildes in \eqref{ewhole} if no confusion is caused.

\subsection{Proof of Theorem \ref{th1}}

First we state two useful lemmas. The first one is the Caccioppoli inequality (energy estimates) and the other is the local boundedness estimate for \eqref{ewhole}. Their proofs are standard and thus omitted.
\begin{lemma}[Caccioppoli inequality]
Let $0<r<R<\infty$ and $u$ be a weak solution of \eqref{ewhole}. Then we have
\begin{equation}
        \label{cac}
 \|\na u\|_{L^2(Q_{r})}\leq \f{C}{R-r} \|u\|_{L^2(Q_{R})}+C\sum^n_{i=1}\|f_i\|_{L^{2}(Q_R)},
\end{equation}
where the constant $C$ depends only on $\la$, $\La$, and $n$.
 \end{lemma}

\begin{lemma}[Local boundedness estimate]
Let $p>n+2$ and $u$ be a weak solution of \eqref{ewhole}. Then we have the following
\begin{equation}
        \label{emean}
\|u\|_{L^\infty(Q_{R/2})}\leq CR^{-1-n/2}\|u\|_{L^2(Q_R)}+CR^{1-(n+2)/p}\sum^n_{i=1}\|f_i\|_{L^{p}(Q_R)}
\end{equation}
for any $0<R<\infty$, where the constant $C$ depends only on $\la$, $\La$, $p$, and $n$.
\end{lemma}

Using the local boundedness estimate \eqref{emean} and by setting $R=1/{\s{k}}, k\in\bN,$ and $p=\infty$, we have
\begin{align}
        \label{emean2}
&\sup_{ Q_{\f{1}{2\s{k}}}(t_0,x_0)}|u|\nonumber\\
&\leq C(\s{k})^{\f{n}{2}+1}\|u\|_{L^2(Q_{\f{1}{\s{k}}}(t_0,x_0))}+C(\s{k})^{-1}\sum^n_{i=1}\|f_i\|_{L^{\infty}(Q_{\f{1}{\s{k}}}(t_0,x_0))}.
\end{align}
Note that for any $l\in\bN$, $\p^l_tu$ is a solution of the following equation
\begin{equation}
        \label{eparabolic1}
\p_t \p^l_tu-\p_i(a_{ij}(x) \p_j\p^l_tu)=\p_i\p^l_tf_i(t,x),\q (t,x)\in (-\infty,0]\times\bR^n.
\end{equation}
To be rigorous, here and in the sequel we need to first take the finite-difference quotients and then pass to the limit.
From \eqref{emean2}, we get
\begin{equation}
        \label{emean3}
\begin{aligned}
\q\sup_{ Q_{\f{1}{2\s{k}}}(t_0,x_0)}|\p^l_tu|\leq& C(\s{k})^{\f{n}{2}+1}\|\p^l_tu\|_{L^2(Q_{\f{1}{\s{k}}}(t_0,x_0))}\\
                               &+C(\s{k})^{-1}\sum^n_{i=1}\|\p^l_tf_i\|_{L^{\infty}(Q_{\f{1}{\s{k}}}(t_0,x_0))}.
                               \end{aligned}
\end{equation}
Before proceeding further, we give a useful lemma.
\begin{lemma}\label{ltd}
Let $u$ be a weak solution of \eqref{ewhole} and $\p_t u\in L^2$. For any $0< r<+\infty$ and $S<T$, denote $(S, T)\times B_r(0)$ by $Q$. For $\dl>0$, define $$
Q^\dl:=\bigcup_{z\in Q}Q_\dl(z).
$$
Then we have,
\begin{equation*}
\|\p_t u\|_{L^2(Q)}\leq C\dl\sum^n_{i=1}\|\p_tf_i\|_{L^{2}({Q}^\dl)}
+C\dl^{-1}\Big(\|\na u\|_{L^2({Q}^\dl)}+\sum^n_{i=1}\|f_i\|_{L^2({Q}^\dl)}\Big),
\end{equation*}
where $C$ depends only on $\la, \La$, and $n$.
\end{lemma}
The essential idea of proving Lemma \ref{ltd} comes from \cite{DK:2011JMA}, where the authors considered the case $f_i\equiv 0$. For completeness, we present the proofs here for the case $f_i\not\equiv 0$.

\begin{proof}
Set
\[
 r_k=\sum^k_{l=1}\f{\dl}{2^l},\q s_k=\f{r_k+r_{k+1}}{2},\q k=1,2,\ldots.
\]
Denote
\[
Q^{(0)}:=Q,\q  Q^{(k)}:={Q}^{r_k},\q \text{and}\q \t{Q}^{(k)}:={Q}^{s_k},\q \text{for}\ k=1,2,\ldots.
\]
Let
\[
 A_k=\|\p_t u\|_{L^2({Q}^{(k)})}\q \text{for}\ k=0,1,2,\ldots,
\]
and
\[
 B=\|\na u\|_{L^2({Q}^{\dl})}+\sum^n_{i=1}\|f_i\|_{L^2({Q}^{\dl})}.
\]
Denote by $\psi_k(t,x)$ a smooth function which vanishes near $\partial\t{Q}^{(k)}$ and satisfies
\begin{equation*}
\psi_k(t,x)=1\q \text{in}\  {Q}^{(k)},\q
|\na\psi_k|^2+|\p_t\psi_k|\leq C\lt({2^k}/{\dl}\rt)^2.
\end{equation*}
Testing \eqref{ewhole} with $\p_t u\psi^2_k$, we get
\begin{equation*}
\int_{\t{Q}^{(k)}}(\p_t u)^2\psi^2_k+\int_{\t{Q}^{(k)}}\big(a_{ij}\p_j u+f_i\big)\p_i(\p_t u\psi^2_k)=0.
\end{equation*}
Then by Young's inequality,
\begin{align*}
&\int_{\t{Q}^{(k)}}(\p_t u)^2\psi^2_kdyds\\
=&-\int_{\t{Q}^{(k)}}\left(a_{ij}\p_i \p_t u\p_ju\psi^2_k
+2a_{ij}\psi_k\p_t u\p_j u\p_i\psi_k+f_i\p_i \p_tu\psi^2_k+2f_i \p_t u\psi_k\p_i\psi_k\right)\\
\leq& C\int_{\t{Q}^{(k)}}|\nabla \p_tu||\nabla u|{
\psi^2_k}+\f{1}{4}\int_{\t{Q}^{(k)}}(\psi_k\p_t u)^2+C\int_{\t{Q}^{(k)}}|\nabla u|^2|\nabla\psi_k|^2\\
&+\int_{\t{Q}^{(k)}}\sum^n_{i=1}|f_i||\nabla \p_tu|{\psi^2_k}+\f{1}{4}\int_{\t{Q}^{(k)}}(\psi_k\p_t u)^2+4\int_{\t{Q}^{(k)}}\sum^n_{i=1}|f_i|^2|\nabla\psi_k|^2.
\end{align*}
The above inequality implies that for any $\varepsilon>0$,
\begin{equation}
        \label{2.17}
A_k\leq \ve\|\na\p_t u\|_{L^2(\t{Q}^{(k)})}+C\big(2^k\dl^{-1}+\ve^{-1}\big)B,
\end{equation}
where $C$ depends only on $\La$ and $n$.

Since $\p_t u$ satisfies
\begin{equation*}
\p_t \p_t u-\p_i(a_{ij}(x) \p_j\p_tu)=\p_i\p_tf_i(t,x),
\end{equation*}
by using the Caccioppoli inequality \eqref{cac}, we have
\begin{equation}
        \label{2.18}
\|\na\p_t u\|_{L^2(\t{Q}^{(k)})}\leq C\f{2^k}{\dl} \|\p_tu\|_{L^2({Q}^{(k+1)})}+C\sum^n_{i=1}\|\p_tf_i\|_{L^2({Q}^{(k+1)})}.
\end{equation}
Inserting \eqref{2.18} into \eqref{2.17}, we get
\begin{equation*}
A_k\leq \ve\f{C2^k}{\dl}A_{k+1}+C\ve\sum^n_{i=1}\|\p_tf_i\|_{L^2({Q}^{(k+1)})}
+C\big({2^k}{\dl^{-1}}+\ve^{-1}\big)B.
\end{equation*}
By choosing $\ve=\f{\dl}{3C2^k}$, we obtain that
\begin{equation}
        \label{2.20}
A_k\leq \f{1}{3}A_{k+1}+C\f{\dl}{2^k}\sum^n_{i=1}\|\p_tf_i\|_{L^{2}({Q}^\dl)}+C\f{2^k}{\dl}B.
\end{equation}
Multiplying both sides of \eqref{2.20} by $3^{-k}$ and summing over $k$, we get
\begin{equation*}
\sum^\infty_{k=0}  3^{-k}A_{k}\leq \sum^\infty_{k=1} 3^{-k}A_{k}+ \sum^\infty_{k=0} \sum^n_{i=1}\f{C\dl}{6^k}\|\p_tf_i\|_{L^{2}({Q}^\dl)}
+\f{C}{\dl}\sum^\infty_{k=0}\Big(\f{2}{3}\Big)^kB.
\end{equation*}
Therefore, by absorbing the first summation on the right-hand of the above inequality, we get
\begin{equation*}
A_0\leq C\dl\sum^n_{i=1}\|\p_tf_i\|_{L^{2}({Q}^\dl)}+\f{C}{\dl}B.
\end{equation*}
This proves the lemma.
\end{proof}

Now we continue the proof of Theorem \ref{th1}. For integers $j=1,2,\ldots,k+1,$ consider the domains
\[
\O^1_j=\Big\{(t,x)\big ||x-x_0|<\f{j}{\s{k}},\ t\in (t_0-\f{j}{k},t_0)\Big\},
\]
\[
\O^2_j=\Big\{(t,x)\big ||x-x_0|<\f{j+0.5}{\s{k}},\ t\in(t_0-\f{j+0.5}{k},t_0)\Big\}.
\]
It is easily seen that for $j=1,2,\ldots,k$, $$
\O^1_j\subset\O^2_j\subset\O^1_{j+1}\subset\O^1_{k+1}.
$$

From Lemma \ref{ltd}, we get
\begin{align*}
\|\p_t u\|_{L^2(\O^1_j)}\leq \f{C}{\s{k}}\sum^n_{i=1}\|\p_tf_i\|_{L^{2}(\O^2_j)}+C\s{k}\Big(\|\na u\|_{L^2(\O^2_j)}+\sum^n_{i=1}\|f_i\|_{L^2(\O^2_j)}\Big).
\end{align*}
Therefore,
\begin{equation}
        \label{eiteration}
\|\p_t u\|_{L^2(\O^1_j)}\leq C\s{k}\Big(\|\na u\|_{L^2(\O^2_j)}+\sum^n_{i=1}\|(f_i,k^{-1}\p_tf_i)\|_{L^2(\O^2_j)}\Big).
\end{equation}
Again using \eqref{cac}, we have
\begin{equation}
        \label{caccio1}
\|\na u\|_{L^2(\O^2_{j})}\leq C\s{k}\| u\|_{L^2(\O^1_{j+1})}+C\sum^n_{i=1}\|f_i\|_{L^{2}(\O^1_{j+1})}.
\end{equation}
The above inequalities \eqref{eiteration} and \eqref{caccio1} indicate that
\begin{equation*}
\|\p_t u\|_{L^2(\O^1_j)}\leq C{k}\| u\|_{L^2(\O^1_{j+1})}+C\s{k}\sum^n_{i=1}\|(f_i,k^{-1}\p_t f_i)\|_{L^2(\O^1_{k+1})}.
\end{equation*}
Since $\p^l_t u$ also satisfies \eqref{eparabolic1}, we have
\begin{equation*}
\|\p^{l+1}_tu\|_{L^2(\O^1_j)}\leq C{k}\| \p^l_tu\|_{L^2(\O^1_{j+1})}+C\s{k}\sum^n_{i=1}\|(\p^l_tf_i,k^{-1}\p^{l+1}_t f_i)\|_{L^2(\O^1_{k+1})}.
\end{equation*}
By iterating over $l$ from $k-1$ to $0$, we get
\begin{equation}
        \label{2.27}
\|\p^{k}_tu\|_{L^2(\O^1_1)}\leq C^{k}{k^k}\|u\|_{L^2(\O^1_{k+1})}+\sum^{k}_{l=0}\sum^n_{i=1}
C^{k-l+1}k^{k-l-1/2}\|\p^l_tf_i\|_{L^2(\O^1_{k+1})}.
\end{equation}
Combining \eqref{emean3} and \eqref{2.27}, we have
\begin{align*}
&|\p^k_tu(t_0,x_0)|\\
\leq& C^{k+1} k^{k+\f{n+2}{4}}\|u\|_{L^2(\O^1_{k+1})}+C (\s{k})^{-1}\sum^n_{i=1}\|\p^k_tf_i\|_{L^{\infty}(Q_{\f{1}{\s{k}}})}\\
& +C k^{\f{n+2}{4}}\sum^{k}_{l=0}\sum^n_{i=1}C^{k-l+1}k^{k-l-1/2}\|\p^l_tf_i\|_{L^2(\O^1_{k+1})}.
\end{align*}
Now suppose that $u$ and $f$ satisfy the bounds in \eqref{ebu} and \eqref{ebc}. Substituting these into the above inequality and using H\"older's inequality, we get
\begin{equation}\label{etds}
|(\p^k_tu)(t_0,x_0)|\leq A_1A^{k+1}_3e^{2A_2|x_0|^2}  k^{k}
\end{equation}
for some $A_3$ depending only on $n$, $A_2$, $\la$, and $\La$. Fixing an $R>0$, for any $x\in B_R(0)$, $t_0\in[-1,0]$, and any $j\in \bN$, by Taylor's formula, we have
\begin{equation}\label{eTaylor}
u(t,x)-\sum^{j-1}_{i=0}\p^i_t u(t_0,x)\f{(t-t_0)^i}{i!}=\p^j_t u(\tau,x)\f{(t-t_0)^j}{j!},
\end{equation}
where $\tau=\tau(t,t_0,x,j)$ lies between $t$ and $t_0$. When $|t-t_0|\leq \dl$ for a sufficiently small $\dl>0$, we see that the right-hand side of \eqref{eTaylor} converges to zero when $j$ goes to infinity, which means that $u(t,x)$ is analytic in time at $t=t_0$ with the radius $\dl$. Hence by letting $j\rightarrow\infty$, we have
\begin{equation*}
u(t,x)=\sum^{\infty}_{j=0} d_j(t_0,x)\f{(t-t_0)^j}{j!}\q {\rm for}\ |t-t_0|\leq \dl,\ x\in B_R(0),\ t\in[-1,0]
\end{equation*}
with $d_j(t_0,x)=\p^j_t u(t_0,x)$ satisfying
\begin{equation*}
|d_j(t_0,x)|\leq A_1 A^{k+1}_3e^{2A_2|x|^2}  j^j
\end{equation*}
for some $A_3$ depending only on $n$, $A_2$, $\la$, and $\La$.

\section{The Navier-Stokes equations in the half space}

\subsection{Bounded mild solutions in the half space}

Let us recall some basic facts from \cite{KS:2002IMS, Sva:2003JMS, Sva:2003RMS} about the initial-boundary problem of the linear Stokes equations with $u=(u_1,\ldots,u_n):$ $(0,\infty)\times \bR^n_+\rightarrow\bR^n$:
\begin{equation}
        \label{es}
\left\{
\begin{aligned}
&\p_t u+\na p-\Dl u=f,\quad \Div u=0,\\
&u(0,x)=u_0(x),\\
&u(t,x)|_{x_n=0}=0,
\end{aligned}
\rt.
\end{equation}
where $u_0(x)$ satisfies
\begin{equation*}
\na\cdot u_0(x)=0, \q u_0|_{x_n=0}=0.
\end{equation*}
First let us assume that $f(t,x)$ is a smooth vector field and decays sufficiently fast at infinity. We decompose $f(t,x)$ into the gradient and solenoidal parts by the standard formula
\begin{equation*}
f(t,x)=\na \Phi(t,x)+\bQ f(t,x)=\bP f+\bQ f,
\end{equation*}
where
\begin{equation*}
\Phi(t,x)=-\int_{\bR^n_+}\na_yN(x,y)\cdot f(t,y)\ dy,
\end{equation*}
\begin{equation*}
N(x,y)=E(x-y)+E(x-y^\ast)
\end{equation*}
is the Green function of the Neumann problem for the Laplace operator in the half space, and
\begin{equation*}
E(z)=\lt\{
\begin{aligned}
&-\f{1}{n(n-2)\al(n)}|z|^{2-n},\q n>2,\\
&\f{1}{2\pi}\ln |z|,\qq\qq\qq\q n=2
\end{aligned}
\rt.
\end{equation*}
is the fundamental solution of the Laplace operator. Here $\al(n)$ denotes the volume of the unit ball in $\bR^n$ and $y^\ast=(y',-y_n)$. Thus,
\begin{equation*}
\bQ f(t,x)=f(t,x)+\na_x\int_{\bR^n_+}\na_yN(x,y)\cdot f(t,y)\ dy
\end{equation*}
satisfies the condition
\begin{equation*}
\Div \bQ f=0,\q  (\bQ f)_n|_{x_n=0}=0.
\end{equation*}

The solution of \eqref{es} has the representation formula
\begin{equation*}
u_i(t,x)=\int_{\bR^n_+}G_{ij}(t;x,y)u_{0j}(y)\ dy+\int^t_0\int_{\bR^n_+}G_{ij}(t-\tau;x,y)(\bQ f)_j(\tau,y)\ dyd\tau.
\end{equation*}
See, for instance, Page 1727 of \cite{Sva:2003JMS}. The elements of the matrix $G$ are given by
\begin{align*}
G_{ij}(t;x,y)=&\dl_{ij}(\G(t,x-y)-\G(t,x-y^\ast))\\
              &+4(1-\dl_{jn})\p_{x_j}\int^{x_n}_0\int_{\bR^{n-1}}\p_{x_i} E(x-z)\G(t,z-y^\ast)dz,
\end{align*}
where $\G(t,x)$ is the fundamental solution of the heat equation. Moreover, if we write $G_{ij}$ as
\begin{equation}
        \label{3.10e}
G_{ij}=\dl_{ij}\G(t,x-y)+G^\ast_{ij},
\end{equation}
then $G^\ast_{ij}$ have the following estimate.

\begin{proposition}[Equation (2.38) of \cite{Sva:2003RMS}]\label{pro1}
For any multi-indices $m=(m',m_n)$ and $k=(k',k_n)$, the kernel functions $G^\ast_{ij}(t;x,y)$ satisfy
\begin{align*}
&|\p^s_t\p^k_x\p^m_yG^\ast_{ij}(t;x,y)|\\
&\leq C^s s^{s} C_{k,m}t^{-s-\f{m_n}{2}}(t+x^2_n)^{-\f{k_n}{2}}
(|x-y^\ast|^2+t)^{-\f{n+|k'|+|m'|}{2}}e^{-\f{cy^2_n}{t}},
\end{align*}
where $C$ is a constant independent of $s$ and $C_{k,m}$ depends only on $k$ and $m$.
\end{proposition}
\begin{remark}
Equation (2.38) of \cite{Sva:2003RMS} contains a less explicit constant $C_{s,k,m}$, which depends on $k,m$ and $s$. By an inspection of the proof in \cite{Sva:2003RMS}, we see that it can be replaced with $C^s s^{s} C_{k,m}$.
\end{remark}

Now in order to define mild solutions of the Navier-Stokes equations, we assume that the source term $f$ is in the divergence form $f=\na\cdot F$,
where $F=(F_{ij})_{1\leq i,j\leq n}$ with $F_{nm}|_{x_n=0}=0$ for $m=1,2,\ldots,n$. The $j$-th component of $\na\cdot F$ is given by
$$
[\na\cdot F]_j=\sum^n_{i=1}\p_iF_{ij}.
$$

Next we give the representation of $\bQ(\na\cdot F)$. As was shown in Proposition 3.1 of \cite{KS:2002IMS},
$$
\bQ(\na\cdot F)=\na\cdot F',
$$
where
\begin{equation}
        \label{eFprime}
\begin{aligned}
F'_{km}=&F_{km}-\dl_{km}F_{nn}+(1-\dl_{nk})\p_{x_m}\Big[\sum^n_{q=1}\int_{\bR^n_+}\p_{y_q} N(x,y)F_{kq}(t,y)\ dy\\
        &+\int_{\bR^n_+}\big(\p_{y_n} N(x,y)F_{nk}(t,y)-\p_{y_k} N(x,y)F_{nn}(t,y)\big)\ dy\Big].
\end{aligned}
\end{equation}
Before proceeding further, we give some notation and equalities. Denote
$$
N^{-}(x,y)=E(x-y)- E(x-y^\ast),
$$
which is the Green function of the Dirichlet problem for the Laplace operator in the half space. Let $\beta$ and $\g$ take values in $\{1,2,\ldots,n-1\}$, and $i$, $j$, $k$, $m$, and $q$ take values in $\{1,2,\ldots,n\}$. It is easily seen that
\begin{equation}
        \label{eN}
\begin{aligned}
\p_{y_\g} N(x,y)&=-\p_{x_\g} N(x,y),\\
\p_{y_\g}N^{-}(x,y)&=-\p_{x_\g} N^{-}(x,y),\\
\p_{y_n} N(x,y)&=-\p_{x_n} N^{-}(x,y).
\end{aligned}
\end{equation}
Now from \eqref{eFprime}, we have
\begin{equation}
        \label{eFprime1}
F'_{nm}=F_{nm}-\dl_{nm}F_{nn},\q m=1,2,\ldots,n.
\end{equation}
Using \eqref{eFprime} and \eqref{eN}, we have
\begin{equation}
        \label{eFprime2}
\begin{aligned}
F'_{\beta\g}=&F_{\beta\g}-\dl_{\beta\g}F_{nn}+\p_{x_\g}\bigg(\sum^n_{q=1}\int_{\bR^n_+}\p_{y_q} N(x,y)F_{\beta q}(t,y)\ dy\\
&+\int_{\bR^n_+}\big(\p_{y_n} N(x,y)F_{n\beta}(t,y)-\p_{y_\beta} N(x,y)F_{nn}(t,y)\big)\ dy\bigg)\\
=&F_{\beta\g}-\dl_{\beta\g}F_{nn}-\p_{x_\g}\sum^{n-1}_{q=1}\int_{\bR^n_+}\p_{x_q} N(x,y)F_{\beta q}(t,y)\ dy\\
&-\p_{x_\g}\int_{\bR^n_+}\p_{x_n} N^-(x,y)F_{\beta n}(t,y)\ dy\\
&-\p_{x_\g}\int_{\bR^n_+}\p_{x_n} N^-(x,y)F_{n\beta}(t,y)+\p_{x_\g}\int_{\bR^n_+}\p_{x_\beta} N(x,y)F_{nn}(t,y)\ dy.
\end{aligned}
\end{equation}
Using the properties of $E(x,y)$, we rewrite $F'_{\beta n}$ as follows.
From \eqref{eFprime} and \eqref{eN}, we have
\begin{footnotesize}
\begin{equation}
        \label{eFprime3}
\begin{aligned}
F'_{\b n}=&F_{\b n}+\p_{x_n}\Big[\sum^n_{q=1}\int_{\bR^n_+}\p_{y_q} N(x,y)F_{\b q}(t,y)\ dy\\
        &+\int_{\bR^n_+}\big(\p_{y_n} N(x,y)F_{n\b}(t,y)-\p_{y_\b} N(x,y)F_{nn}(t,y)\big)\ dy\Big]\\
=&F_{\b n}-\sum^{n-1}_{\g=1}\p_{x_\g}\int_{\bR^n_+}\p_{x_n} N(x,y)F_{\b\g}(t,y)\ dy+\p_{x_\b}\int_{\bR^n_+}\p_{x_n} N(x,y)F_{nn}(t,y)\ dy\\
&-\p^2_{x_n}\int_{\bR^n_+} N^-(x,y)(F_{\b n}+F_{n\b})(t,y)\ dy\\
=&F_{\b n}-\sum^{n-1}_{\g=1}\p_{x_\g}\int_{\bR^n_+}\p_{x_n} N(x,y)F_{\b\g}(t,y)\ dy+\p_{x_\b}\int_{\bR^n_+}\p_{x_n} N(x,y)F_{nn}(t,y)\ dy\\
&+(-\Dl_x+\sum^{n-1}_{\g=1}\p^2_{x_\g})\int_{\bR^n_+} N^-(x,y)(F_{\b n}+F_{n\b})(t,y)\ dy\\
=&F_{\b n}-\sum^{n-1}_{\g=1}\p_{x_\g}\int_{\bR^n_+}\p_{x_n} N(x,y)F_{\b\g}(t,y)\ dy+\p_{x_\b}\int_{\bR^n_+}\p_{x_n} N(x,y)F_{nn}(t,y)\ dy\\
&-(F_{\b n}+F_{n\b})+\sum^{n-1}_{\g=1}\p^2_{x_\g}\int_{\bR^n_+} N^-(x,y)(F_{\b n}+F_{n\b})(t,y)\ dy.
\end{aligned}
\end{equation}
\end{footnotesize}
Here in the second equality, we used \eqref{eN}, and in the fourth equality, we used that $N^-$ is the Green function of the Dirichlet problem for the Laplace operator in the half space.
Thus, inserting \eqref{eFprime1}, \eqref{eFprime2} and \eqref{eFprime3} into $\na\cdot F'$ and by a simply calculation, we get
\begin{equation*}
[\bQ(\na\cdot F)]_j=\na\cdot F'=\sum^n_{k=1}\p_{x_k} F_{kj}-\p_{x_j} F_{nn}-\sum^{n-1}_{\beta=1}\p_{x_\beta}(F_{\beta n}+F_{n\beta})\dl_{nj}+h_j,
\end{equation*}
where $h_j$ has the form
\begin{equation}\label{ehj}
h_j=\sum_{\beta,\g,q,k,l}C_{j\beta\g q k l}\p^2_{x_\beta x_\g}\int_{\bR^n_+}\p_{x_q} N^\pm(x,y)F_{kl}(t,y)\ dy
\end{equation}
and $C_{j\beta\g q k l}$ are constants.
Now we use integration by parts to obtain
\begin{footnotesize}
\begin{equation*}
\begin{aligned}
&\int^t_0\int_{\bR^n_+}G_{ij}(t-\tau;x,y)[\bQ(\na\cdot F)]_j(\tau,y)\ dyd\tau\\
=&\int^t_0\int_{\bR^n_+}G_{ij}(t-\tau;x,y)\Big[\p_{ y_k} F_{kj}-\p_{ y_j} F_{nn}-\p_{y_\beta}(F_{\beta n}+F_{n\beta})\dl_{nj}+h_j\Big](\tau,y)\ dyd\tau\\
=&-\int^t_0\int_{\bR^n_+}\p_{y_k} G_{ij}(t-\tau;x,y)F_{kj}(\tau,y)\ dyd\tau+\int^t_0\int_{\bR^n_+}\p_{y_j} G_{ij}(t-\tau;x,y)F_{nn}(\tau,y)\ dyd\tau\\
&+\int^t_0\int_{\bR^n_+}\p_{y_\beta} G_{in}(t-\tau;x,y)(F_{\beta n}+F_{n\beta})(\tau,y)\ dyd\tau\\
&+\underbrace{\int^t_0\int_{\bR^n_+}G_{ij}(t-\tau;x,y)h_j(\tau,y)\ dyd\tau}_{I_i}.
\end{aligned}
\end{equation*}
\end{footnotesize}
By \eqref{ehj}, we further calculate $I_i$ as
\begin{footnotesize}
\begin{equation*}
\begin{aligned}
I_i=&\sum_{j,\beta,\g,q,k,l}C_{j\beta\g q k l}\int^t_0\int_{\bR^n_+}G_{ij}(t-\tau;x,y)\p^2_{y_\beta y_\g}\int_{\bR^n_+}\p_{y_q} N^\pm(y,z)F_{kl}(\tau,z)\ dzdyd\tau\\
=&\sum_{j,\beta,\g,q,k,l}C_{j\beta\g q k l}\int^t_0\int_{\bR^n_+}G_{ij}(t-\tau;x,z)\p^2_{z_\beta z_\g}\int_{\bR^n_+}\p_{z_q}N^\pm(z,y)F_{kl}(\tau,y)\ dydzd\tau\\
=&\sum_{j,\beta,\g,q,k,l}C_{j\beta\g q k l}\int^t_0\int_{\bR^n_+}G_{ij}(t-\tau;x,z)\int_{\bR^n_+}\lt[\p^2_{y_\beta y_\g}\p_{z_q} N^\pm(z,y)\rt]F_{kl}(\tau,y)\ dydzd\tau\\
=&\sum_{j,\beta,\g,q,k,l}C_{j\beta\g q k l}\int^t_0\int_{\bR^n_+}\lt[\p^2_{y_\beta y_\g}\int_{\bR^n_+}G_{ij}(t-\tau;x,z)\p_{z_q} N^\pm(z,y)dz\rt]F_{kl}(\tau,y)\ dyd\tau.
\end{aligned}
\end{equation*}
\end{footnotesize}
Set
\begin{equation*}
K_{ijq}(t;x,y)=\int_{\bR^n_+}G_{ij}(t;x,z)\p_{z_q} N^\pm(z,y)\ dz.
\end{equation*}
Then we have
\begin{equation*}
I_i=\sum_{j,\beta,\g,q,k,l}C_{j\beta\g q k l}\int^t_0\int_{\bR^n_+}\p^2_{y_\beta y_\g}K_{ijq}(t-\tau;x,y)F_{kl}(\tau,y)\ dyd\tau.
\end{equation*}
\begin{proposition}[Proposition 3.1 of \cite{Sva:2003JMS}]\label{pro2}
The kernels $K_{ijq}$ satisfy
\begin{equation*}
|\p^s_t\p^{k'}_{y'}K_{ijq}(t;x,y)|\leq C^ss^{s}C_{k'}t^{-s}(|x-y|^2+t)^{-\f{n-1+|k'|}{2}}.
\end{equation*}
\end{proposition}

Now if the source term on the right-hand side of \eqref{es} is $f=\na\cdot F$, the solution is given by
\begin{footnotesize}
\begin{equation}
        \label{elm}
\begin{aligned}
u_i(t,x)=&\int_{\bR^n_+}G_{ij}(t;x,y)u_{0j}(y)\ dy+\int^t_0\int_{\bR^n_+}
G_{ij}(t-\tau;x,y)[\bQ (\na\cdot F)]_j(\tau,y)\ dyd\tau\\
=&\int_{\bR^n_+}G_{ij}(t;x,y)u_{0j}(y)\ dy-\sum^n_{k,j=1}\int^t_0\int_{\bR^n_+}\p_{y_k} G_{ij}(t-\tau;x,y)F_{kj}(\tau,y)\ dyd\tau\\
&+\sum^n_{j=1}\int^t_0\int_{\bR^n_+}\p_{y_j} G_{ij}(t-\tau;x,y)F_{nn}(\tau,y)\ dyd\tau\\
&+\sum^{n-1}_{\beta=1}\int^t_0\int_{\bR^n_+}\p_{y_\beta} G_{in}(t-\tau;x,y)(F_{\beta n}+F_{n\beta})(\tau,y)\ dyd\tau\\
&+\sum_{j,\beta,\g,q,k,l}C_{j\beta\g q k l}\int^t_0\int_{\bR^n_+}\p^2_{y_\beta y_\g}K_{ijq}(t-\tau;x,y)F_{kl}(\tau,y)\ dyd\tau.
\end{aligned}
\end{equation}
\end{footnotesize}
For the incompressible Navier-Stokes equations in the half space
\begin{equation}
        \label{ensh}
\left\{
\begin{aligned}
&\p_t u+u\cdot\na u-\Dl u+\na p=\na\cdot F,\\
&\Div u=0,\\
&u(0,x)=u_0(x),\\
&u(t,x)|_{x_n=0}=0,
\end{aligned}
\rt.
\end{equation}
where $F=(F_{ij})_{1\leq i,j\leq n}$ with $F_{nm}|_{x_n=0}=0$ for $m=1,2,\ldots,n$,
we can rewrite the first equation as
\begin{equation*}
\p_t u-\Dl u+\na p=-\na\cdot(u\otimes u)+\na\cdot F=\na\cdot(F-u\otimes u).
\end{equation*}
Substituting $F_{kl}$ with $F_{kl}-u_ku_l$ in \eqref{elm}, we finally get the formulation of mild solutions to \eqref{ensh}:
\begin{equation}
        \label{milds}
\begin{aligned}
&u_i(t,x)\\
=&\int_{\bR^n_+}G_{ij}(t;x,y)u_{0j}(y)\ dy\\
&+\sum^n_{k,j=1}\int^t_0\int_{\bR^n_+}{\p_{y_k}}
G_{ij}(t-\tau;x,y)(u_ku_j-F_{kj})(\tau,y)\ dyd\tau\\
&-\sum^n_{j=1}\int^t_0\int_{\bR^n_+}{\p_{y_j}}
G_{ij}(t-\tau;x,y)(u^2_n-F_{nn})(\tau,y)\ dyd\tau\\
&-\sum^{n-1}_{\beta=1}\int^t_0\int_{\bR^n_+}\p_{y_\beta} G_{in}(t-\tau;x,y)(2u_\beta u_n-F_{\beta n}-F_{n\beta})(\tau,y)\ dyd\tau\\
&-\sum_{j,\beta,\g,q,k,l}C_{i\beta\g q k l}\int^t_0\int_{\bR^n_+}\p^2_{y_\beta y_\g}K_{ijq}(t-\tau;x,y)(u_ku_l-F_{kl})(\tau,y)\ dyd\tau.
\end{aligned}
\end{equation}
For the equation \eqref{milds}, we have the following local solvability result.
\begin{proposition}
For any bounded $u_0\in C(\bR^n_+)$ satisfying $\na\cdot u_0=0$, $u_0|_{x_n=0}=0$ and bounded $F(t,x)\in C((0,\infty)\times {\overline{\bR^n_+}})$ with $F_{nm}|_{x_n=0}=0$ for $m=1,2,\ldots,n$, there exits a $T>0$ such that \eqref{milds} has a unique solution $u\in C((0,T)\times \bR^n_+)$.
\end{proposition}
\begin{proof}
For simplification, we ignore the indices, signs, and constants in the formula \eqref{milds} if no confusion is caused and write it as
\begin{equation}
        \label{milds3}
\begin{aligned}
u(t,x)=&\int_{\bR^n_+}G(t;x,y)u_{0}(y)\ dy+\int^t_0\int_{\bR^n_+}{\na_y G}(t-\tau;x,y)(u^2+F)(\tau,y)\  dyd\tau\\
&+\int^t_0\int_{\bR^n_+}\na^2_{y'}K(t-\tau;x,y)(u^2+F)(\tau,y)\ dyd\tau\\
=&\int_{\bR^n_+}G(t;x,y)u_{0}(y)\ dy+\int^t_0\int_{\bR^n_+}\t{K}(t-\tau;x,y)(u^2+F)(\tau,y)\ dyd\tau,
\end{aligned}
\end{equation}
where $\t{K}:={\na_y G}+\na^2_{y'}K$.
In order to solve the equation \eqref{milds3}, we define the function space 
\begin{align*}
&C_b((0,T)\times \bR^n_+)\\
&:=\lt\{u\in C((0,T)\times \bR^n_+):\ |u|\leq C_\ast (\|u_0\|_{L^\i(\bR^n_+)}+\|F\|_{L^\i((0,T)\times \bR^n_+)}) \rt\},
\end{align*}
where $C_\ast$ will be determined later.

Define
$$
u^0(t,x):= \int_{\bR^n_+}G(t;x,y)u_{0}(y)\ dy.
$$
From Proposition \ref{p3.4} below for the estimate of $G$, we have
\begin{equation*}
\|u^0\|_{L^\i((0,T)\times\bR^n_+)}\leq \|u_0\|_{L^\i(\bR^n_+)}\|G(t;x,y)\|_{L^1_y}\leq C_0 \|u_0\|_{L^\i(\bR^n_+)}.
\end{equation*}
We then solve the equation \ref{milds3} by the method of successive approximation:
\begin{equation}\label{milds4}
\begin{aligned}
u^{m+1}(t,x)=&u^0(t,x)+\int^t_0\int_{\bR^n_+}\t{K}(t-\tau;x,y)((u^{m})^2+F)(\tau,y)\ dyd\tau,
\end{aligned}
\end{equation}
for $m=0,1,\ldots$, in $C_b((0,T)\times \bR^n_+)$ by choosing $C_\ast=2C_0$ and a suitably small $T$.

Indeed, if
$$
|u^m(t,x)|\leq 2C_0 (\|u_0\|_{L^\i(\bR^n_+)}+\|F\|_{L^\i((0,T)\times \bR^n_+)}),
$$
then from \eqref{milds4} and Proposition \ref{p3.4} for the estimates of $\na_y G$ and $\na^2_{y'} K$, we have for any $(t,x)\in (0,T)\times \bR^n_+$,
\begin{equation*}
\begin{aligned}
&|u^{m+1}(t,x)|\\
\leq &\|u^0\|_{L^\i((0,T)\times\bR^n_+)}+ \|((u^{m})^2+|F|)\|_{L^\i((0,T)\times\bR^n_+)}\int^t_0\|\t{K}(t-\tau;x,y)\|_{L^1_y}d\tau\\
 \leq & C_0 \|u_0\|_{L^\i(\bR^n_+)}+\lt(8C^2_0 (\|u_0\|^2_{L^\i(\bR^n_+)}+ \|F\|^2_{L^\i((0,T)\times\bR^n_+)} )+ \|F\|_{L^\i((0,T)\times\bR^n_+)}\rt)\\
 &\quad \cdot\int^t_0\f{C}{\s{t-\tau}}d\tau\\
 \leq & C_0 \|u_0\|_{L^\i(\bR^n_+)}+C\s{t}\lt(8C^2_0 (\|u_0\|^2_{L^\i(\bR^n_+)}+ \|F\|^2_{L^\i((0,T)\times\bR^n_+)} )+ \|F\|_{L^\i((0,T)\times\bR^n_+)}\rt)\\
 \leq & 2C_0 (\|u_0\|_{L^\i(\bR^n_+)}+\|F\|_{L^\i((0,T)\times\bR^n_+)})
\end{aligned}
\end{equation*}
provided that $T$ is sufficiently small.
This shows that $u^{m+1}\in C_b((0,T)\times \bR^n_+)$. Moreover, from \eqref{milds4}, it is easily seen that
\begin{equation*}
\begin{aligned}
&|u^{m+1}-u^m|\\
 \leq &\lt |\int^t_0\int_{\bR^n_+}\t{K}(t-\tau;x,y)((u^{m})^2-(u^{m-1})^2)(\tau,y)\ dyd\tau\rt|\\
 \leq &  \|(u^{m})^2-(u^{m-1})^2\|_{L^\i((0,T)\times\bR^n_+)}\lt |\int^t_0\int_{\bR^n_+}\t{K}(t-\tau;x,y)\ dyd\tau\rt|\\
 \leq & \|u^{m}-u^{m-1}\|_{L^\i((0,T)\times\bR^n_+)} \|u^{m}+u^{m-1}\|_{L^\i((0,T)\times\bR^n_+)}\lt |\int^t_0\|\t{K}(t-\tau;x,y)\|_{L^1_y}d\tau\rt|\\
 \leq & 4C_0 (\|u_0\|_{L^\i(\bR^n_+)}+\|F\|_{L^\i((0,T)\times\bR^n_+)}) \|u^{m}-u^{m-1}\|_{L^\i((0,T)\times \bR^n_+)}C\s{t}\\
 \leq & \f{1}{2} \|u^{m}-u^{m-1}\|_{L^\i((0,T)\times \bR^n_+)}
\end{aligned}
\end{equation*}
by choosing a sufficiently small $T$ such that
\begin{equation*}
4CC_0 (\|u_0\|_{L^\i(\bR^n_+)}+\|F\|_{L^\i((0,T)\times\bR^n_+)})\s{T}\leq 1/2.
\end{equation*}
Finally the Banach fixed point theorem indicates that the sequence $u^m$ has a unique limit which is a solution of \eqref{milds3}.
\end{proof}

%

\subsection{Time analyticity}
The main result of this section is the following theorem.
\begin{theorem}\label{thns}
Let $C_0>0$ be a constant. Assume that $u$ is a bounded mild solution of the Navier-Stokes equations \eqref{ensh} on $[0,1]\times\bR^n_+$ with $|u|\leq C_0$ and $F(t,x)$ satisfies
\begin{equation}
        \label{fassum}
\|t^k\p^k_t F\|_{L^\infty((0,\infty)\times \bR^n_+)}\leq  C^{k+1}_0k^k
\end{equation}
and $F_{nm}|_{x_n=0}=0$ for $m=1,2,\ldots,n$.
Then for any $k\geq 1$, we have
\begin{equation*}
\sup_{t\in(0,T]}t^k\|\p^k_t u(t,\cdot)\|_{L^\infty(\bR^n_+)}\leq  M^{k+1}k^k
\end{equation*}
for a sufficiently large constant $M$. Consequently, $u(t,x)$ is analytic in time for $t\in(0,1]$.
\end{theorem}

The proof of Theorem \eqref{thns} relies on taking time derivatives of the integral representation \eqref{milds}. Here a difficult is that the time derivatives of the kernel $G_{ij}$ and $K_{ijq}$ are not locally integrable in space time near the origin, let alone the high order derivatives. We will follow the technique used in \cite{DZ:2019arxiv1} to allow differentiation. The following lemmas will be used frequently.
\begin{lemma}[Lemma 3.1 of \cite{DZ:2019arxiv1}]\label{lid1}
For any $k\geq 1$, we have
\begin{equation*}
\sum^{k-1}_{j=1}\binom{k}{j}j^{j-2/3}(n-j)^{n-j-2/3}\leq Ck^{k-2/3},
\end{equation*}
where $C$ is a constant independent of $k$.
\end{lemma}
\begin{lemma}[Lemma 3.2 of \cite{DZ:2019arxiv1}]\label{lid2}
Let $f$ and $g$ be two smooth functions on $\bR$. For any integer $k\geq 1$, we have
\begin{align*}
\p^k_t(t^kf(t)g(t))=&\sum^{k}_{j=0}\binom{k}{j} \p^j_t(t^jf(t))\p^{k-j}_t(t^{k-j}g(t))\\
                      &-k\sum^{k-1}_{j=0}\binom{k-1}{j} \p^j_t(t^jf(t))\p^{k-1-j}_t(t^{k-1-j}g(t)).
\end{align*}
\end{lemma}
By applying Propositions \ref{pro1} and \ref{pro2}, we deduce the following estimates for the kernels $G_{ij}$ and $K_{ijq}$.
\begin{proposition}\label{p3.4}
There exists a constant $C$, independent $k$, such that for any $k\in \bN\cup \{0\}$,
\begin{equation}
        \label{e3.32}
\|t^{k}\p^k_tG_{ij}(t;x,y)\|_{L^1_y}\leq C^{k+1}k^{k},\q \|t^{k}\p^k_t\na_y G_{ij}(t;x,y)\|_{L^1_y}\leq C^{k+1}k^{k}t^{-1/2},
\end{equation}
\begin{equation*}
\|t^{k}\p^k_t \p^2_{y_\beta y_\g}K_{ijq}(t;x,y)\|_{L^1_y}\leq C^{k+1}k^{k}t^{-1/2}.
\end{equation*}
\end{proposition}
\begin{proof}
We only prove the first inequality of \eqref{e3.32} as the others are essentially the same. From \eqref{3.10e}
\begin{equation*}
\|t^k\p^k_tG_{ij}(t;x,y)\|_{L^1_y}\leq \|t^k\p^k_t\G(t;x,y)\|_{L^1_y}+\|t^k\p^k_tG^\ast_{ij}(t;x,y)\|_{L^1_y}.
\end{equation*}
Since
\begin{equation*}
\G(t;x,y)=
\begin{aligned}
(4\pi t)^{-n/2}\exp\lt(-\f{|x-y|^2}{4t}\rt)\q \forall t> 0,\\
\end{aligned}
\end{equation*}
by using the semi-group property of $\Gamma$, for any $t>0$,
\begin{equation*}
|\p^k_t\G(t;x,y)|\leq C^{k}k^{k}t^{-k-n/2}\exp\lt(-\f{|x-y|^2}{8t}\rt).
\end{equation*}
Thus a direct calculation shows that
\begin{equation}
        \label{Gamma1}
\|t^k\p^k_t\G(t;x,y)\|_{L^1_y}\leq CC^{k}k^{k}.
\end{equation}
Using Proposition \ref{pro1}, we have
\begin{equation*}
\begin{aligned}
&\|t^{k}\p^k_tG^\ast_{ij}(t;x,y)\|_{L^1_y}\\
\leq& C^kk^{k}\int_{\bR^n_+}(|x-y^\ast|^2+t)^{-n/2}e^{-\f{cy^2_n}{t}}\ dy\\
\leq& C^kk^{k}\int_{\bR^n_+}(|x'-y'|^2+|x_n+y_n|^2+t)^{-n/2}e^{-\f{cy^2_n}{t}}\ dy_ndy'\\
\leq& C^kk^{k}\int_{\bR^n_+}(|x'-y'|^2+y^2_n+t)^{-n/2}e^{-\f{cy^2_n}{t}}\ dy_ndy'.
\end{aligned}
\end{equation*}
Here in the last line, we used the fact that $x_n+y_n\geq y_n$ since $x,y\in\bR^n_+$. Making a change of variables $x'-y':=\t{y}'$ and still denoting $\t{y}'$ by $y'$, we have
\begin{equation}
        \label{kerne1}
\begin{aligned}
&\|t^{k}\p^k_tG^\ast_{ij}(t;x,y)\|_{L^1_y}\\
\leq& C^kk^{k}\int_{\bR^n_+}(|y'|^2+y^2_n+t)^{-n/2}e^{-\f{cy^2_n}{t}}\ dy_ndy'\q (y\rightarrow y\s{t})\\
=& C^kk^{k}\int_{\bR^n_+}(|y'|^2+y^2_n+1)^{-n/2}e^{-cy^2_n}\ dy_ndy'\\
\leq& C^kk^{k}\int_{\bR^n_+}(|y'|^2+1)^{-n/2}\ dy'
\leq CC^kk^{k}.
\end{aligned}
\end{equation}
A combination of \eqref{Gamma1} and \eqref{kerne1} proves the first inequality of \eqref{e3.32}.
\end{proof}

Theorem \ref{thns} is a direct consequence of the following proposition.

\begin{proposition}\label{p3.5}
Under the assumptions of Theorem \ref{thns}, for any $k\geq 1$, we have
\begin{equation}
        \label{ee3.21}
\sup_{t\in(0,T]}\|\p^k_t(t^ku(t,\cdot))\|_{L^\infty(\bR^n)}\leq M^{k-1/2}k^{k-2/3}
\end{equation}
for a sufficiently large $M\geq 1$.
\end{proposition}
\begin{proof}
For simplicity, we still abbreviate \eqref{milds} by \eqref{milds3}. From Proposition \ref{p3.4}, we have for any $j\in \bN\cap\{0\}$,
\begin{equation*}
\|t^j\p^j_t \t{K}(t;x,y)\|_{L^1_y}\leq C^{j+1}j^{j}t^{-1/2}.
\end{equation*}
Using the Leibniz rule and Lemma \ref{lid1}, we have
\begin{equation}
        \label{etK}
\|\p^j_t (t^j\t{K}(t;x,y))\|_{L^1_y}\leq C^{j+1}j^{j}t^{-1/2}
\end{equation}
and
\begin{equation}
        \label{3.45}
\|\p^j_t (t^kG(t;x,y))\|_{L^1_y}\leq C^{j+1}j^{j}.
\end{equation}

Now we prove Proposition \ref{p3.5} by induction. Using \eqref{milds3}, we obtain
\begin{align*}
\p^k_t(t^k u(t,x))&=\p^k_t\bigg(\int_{\bR^n_+}t^kG(t;x,y)u_{0}(y)\ dy\bigg)\\
&\quad +\p^k_t\bigg(t^k\int^t_0\int_{\bR^n_+}\t{K}(t-\tau;x,y)F(\tau,y)\ dyd\tau\bigg)\\
&\quad
+\p^k_t\bigg(t^k\int^t_0\int_{\bR^n_+}\t{K}(t-\tau;x,y)u^2(\tau,y)\ dyd\tau\bigg)\\
&:= I_1+I_2+I_3.
\end{align*}
By using \eqref{3.45}, we have
\begin{equation*}
|I_1|\leq \|u_0\|_{L^\infty}\|\p^k_t(t^k G)\|_{L^1}\leq C_0C^{k+1}k^{k-2/3}\leq M^{k-2/3}k^{k-2/3}
\end{equation*}
for a sufficiently large $M$. To estimate $I_2$, we proceed as follows
\begin{equation}
        \label{di2}
\begin{aligned}
I_2=&\p^k_t\bigg((t-\tau+\tau)^k\int^t_0\int_{\bR^n_+}
\t{K}(t-\tau;x,y)F(\tau,y)\ dyd\tau\bigg)\\
=&\sum^k_{j=0} \binom{k}{j}
\p^k_t\int^t_0\int_{\bR^n_+}(t-\tau)^j\t{K}(t-\tau;x,y)\tau^{k-j}F(\tau,y)\ dyd\tau\\
=&\sum^k_{j=0} \binom{k}{j}\p^{k-j}_t\int^t_0\int_{\bR^n_+}\p^j_t\big((t-\tau)^j
\t{K}(t-\tau;x,y)\big)\tau^{k-j}F(\tau,y)\ dyd\tau\\
=&\sum^k_{j=0} \binom{k}{j}
\p^{k-j}_t\int^t_0\int_{\bR^n_+}\p^j_\tau\big(\tau^j
\t{K}(\tau;x,y)\big)(t-\tau)^{k-j}F(t-\tau,y)\ dyd\tau\\
=&\sum^k_{j=0} \binom{k}{j}
\int^t_0\int_{\bR^n_+}\p^j_\tau\big(\tau^j\t{K}(\tau;x,y)\big)\p^{k-j}_t\big((t-\tau)^{k-j}
F(t-\tau,y)\big)\ dyd\tau,
\end{aligned}
\end{equation}
where in the fourth line of the above equality, we made a change of variables $\tau$ to $t-\tau$. From \eqref{fassum}, the Leibniz rule, and Lemma \ref{lid1}, we have,
\begin{equation}
        \label{fassum1}
\|\p^k_t(t^k F)\|_{L^\infty((0,\infty)\times \bR^n_+)}\leq  C^{k}k^{k-1}.
\end{equation}
The above inequality \eqref{fassum1}, \eqref{etK}, and Lemma \ref{lid1} together imply
\begin{equation*}
\begin{aligned}
&|I_2|\\
\leq& \sum^k_{j=0} \binom{k}{j}
\|\p^j_\tau(\tau^j\t{K}(\tau))\|_{L^{1}((0,t]\times \bR^n_+)}
\|\p^{k-j}_t\big((t-\tau)^{k-j}F(t-\tau,y))\|_{L^\infty((0,t]\times \bR^n_+)}\\
\leq& \sum^k_{j=0} \binom{k}{j}\int^t_0\tau^{-1/2}\ d\tau\cdot C^{j+1}j^{j-2/3} C^{k-j+1}(k-j)^{k-j-2/3}\\
\leq& C^{k+1}k^{k-2/3}t^{1/2}\leq M^{k-{2/3}}k^{k-2/3},\q (\text{recalling }t\leq 1).
\end{aligned}
\end{equation*}
For $I_3$, similar to \eqref{di2} we have
\begin{equation}
                \label{ei3}
I_3=\sum^k_{j=0} \binom{k}{ j}\int^t_0\int_{\bR^n_+}\p^j_\tau\big(\tau^j\t{K}(\tau;x,y)\big)\p^{k-j}_t\big((t-\tau)^{k-j}u^2 (t-\tau,y)\big)\ dyd\tau.
\end{equation}
We estimate the right-hand side of \eqref{ei3} by considering $j=0,k$ and $1\leq j\leq k-1$ separately. By the inductive assumption and Lemmas \ref{lid1} and \ref{lid2}, we have
\begin{equation*}
\begin{aligned}
&|\p^{k}_t\big(t^ku^2(t,y)\big)|\\
=&\Big|2u(t,y)\p^k_t(t^ku(t,y))+\sum^{k-1}_{j=1}\binom{k}{j} \p^j_t(t^ju(t,y))\p^{k-j}_t(t^{k-j}u(t,y))\\
                      &-k\sum^{k-1}_{j=0}\binom{k-1}{j} \p^j_t(t^ju(t,y))\p^{k-1-j}_t(t^{k-1-j}u(t,y))\Big|\\
\leq&2C_0 |\p^{k}_t\big(t^ku(t,y)\big)|+M^{k-3/4}k^{k-2/3},
\end{aligned}
\end{equation*}
and for $j=1,2,\ldots,k-1$,
\begin{equation*}
|\p^{j}_t\big(t^ju^2(t,y)\big)|\leq M^{j-3/4}{j^{j-2/3}}
\end{equation*}
for a sufficiently large $M$. With these in hand, similar to \eqref{di2} we get
\begin{equation*}
\begin{aligned}
|I_3|=&\int^t_0\int_{\bR^n_+}\p^k_\tau\big(\tau^k\t{K}(\tau;x,y)\big)
u^2(t-\tau,y)\ dyd\tau\\
&+\int^t_0\int_{\bR^n_+}\t{K}(\tau;x,y)\p^{k}_t\big((t-\tau)^{k}u^2(t-\tau,y)\big)\ dyd\tau\\
&+\sum^{k-1}_{j=1} \binom{k}{j}\int^t_0\int_{\bR^n_+}\p^j_\tau\big(\tau^j\t{K}(\tau;x,y)\big)\p^{k-j}_t\big((t-\tau)^{k-j}
u^2(t-\tau,y)\big)\ dyd\tau\\
\leq&\int^t_0C^{k+1}k^{k-2/3}\tau^{-1/2}C^2_0\ d\tau\\
&+\int^t_0C\tau^{-1/2}\lt[2C_0\|\p^{k}_t\big((t-\tau)^k
u(t-\tau,\cdot)\big)\|_{L^\infty}+M^{k-3/4}k^{k-2/3}\rt]\ d\tau\\
&+\sum^{k-1}_{j=1}\binom{k}{j}\int^t_0C^{j+1}j^{j-2/3}\tau^{-1/2}M^{k-j-3/4}(k-j)^{k-j-2/3}\ d\tau\\
\leq&Ct^{1/2}M^{k-2/3}k^{k-2/3}+2C_0C\int^t_0(t-\tau)^{-1/2}\|\p^{k}_\tau\big(\tau^k
u(\tau,\cdot)\big)\|_{L^\infty}\ d\tau.
\end{aligned}
\end{equation*}
Combining the estimates of $I_1$, $I_2$, and $I_3$, and setting $\Psi(t):=\|\p^{k}_t\big(t^ku(t,\cdot)\big)\|_{L^\infty}$, we get by a second iteration and using Fubini's theorem that
\begin{equation*}
\begin{aligned}
\Psi(t)\leq& CM^{k-2/3}k^{k-2/3}+C\int^t_0(t-\tau)^{-1/2}\Psi(\tau)\ d\tau\\
\leq&  CM^{k-2/3}k^{k-2/3}+C\int^t_0(t-\tau)^{-1/2}\\
&\quad\cdot \lt(M^{k-2/3}k^{k-2/3}+C\int^\tau_0(\tau-s)^{-1/2}\Psi(s)\ ds \rt)\ d\tau\\
\leq& CM^{k-2/3}k^{k-2/3}+C^2\int^t_0(t-\tau)^{-1/2}\int^\tau_0(\tau-s)^{-1/2}\Psi(s)dsd\tau\\
\leq&CM^{k-2/3}k^{k-2/3}+C^2\int^t_0\Psi(s)\underbrace{\lt(\int^t_s(t-\tau)^{-1/2}(\tau-s)^{-1/2}\ d\tau\rt)}_{\text{bounded}}\ ds\\
\leq&CM^{k-2/3}k^{k-2/3}+C^2\int^t_0\Psi(s)\ ds.
\end{aligned}
\end{equation*}
We then conclude \eqref{ee3.21} by applying the Gronwall inequality and taking a sufficiently large $M$. The proposition is proved.
\end{proof}

\begin{proof}[Proof of Theorem \ref{thns}]
Note that
\begin{equation*}
\p^k_t(t^ju)=k\p^{k-1}_t(t^{j-1}u)+t\p^k_t(t^{j-1}u).
\end{equation*}
Letting $j=k$, by Proposition \ref{p3.5} we have
\begin{equation*}
\begin{aligned}
\sup_{t\in(0,T]}\|t\p^k_t(t^{k-1} u)\|_{L^\infty}\leq& M^{k-1/2}k^{k-2/3}+M^{k-3/2}(k-1)^{k-1-2/3}k\\
  \leq& M^{k}(1+1/M)k^k.
\end{aligned}
\end{equation*}
Then by induction, we get for $j=0,1,\ldots, k$,
\begin{equation*}
\sup_{t\in(0,T]}\|t^j\p^k_t(t^{k-j} u)\|_{L^\infty}\leq M^{k}(1+1/M)^jk^k.
\end{equation*}
Taking $j=k$, we obtain
\begin{equation*}
\sup_{t\in(0,T]}\|t^k\p^k_tu(t,\cdot)\|_{L^\infty(\bR^n_+)}\leq M^{k}(1+1/M)^kk^k=(M+1)^k k^k.
\end{equation*}
This finishes the proof of Theorem \ref{thns}.
\end{proof}

\section*{Acknowledgments}  The authors thank Prof. Qi S. Zhang in UC Riverside for helpful discussions on the topic. X. Pan is supported by Natural Science Foundation of Jiangsu Province (No. SBK2018041027) and National Natural Science Foundation of China (No. 11801268).

\bibliographystyle{plain}



\def\cprime{$'$}


\end{document}